\documentclass{article}
\usepackage{graphicx}
\usepackage{amssymb}
\usepackage{amsmath}
\usepackage{accents}
\usepackage{float}
\usepackage{amsthm}
\usepackage{mathtools}
\usepackage[hidelinks]{hyperref}
\hypersetup{colorlinks=false}
\usepackage{natbib}
\usepackage{url}
\usepackage{indentfirst}

\newcommand{\subC}[1]{\accentset{\leftrightsquigarrow}{#1}}
\newtheorem{theorem}{Theorem}[section]
\newtheorem{corollary}{Corollary}[theorem]
\newtheorem{lemma}[theorem]{Lemma}
\theoremstyle{remark}
\newtheorem*{remark}{Remark}

\title{On the Average Resistance of $n$-circuits}
\date{October 2024}
\author{
  Mehdi Nikopour Deilami\\
  \and
  Bohdan Zhelyabovskyy\\
}

\begin{document}
\maketitle

\begin{abstract}
\normalsize
    $n$-circuits are series-parallel networks composed of exactly $n$ unit resistors. This paper is focused on evaluating the mean resistance of all $n$-circuits, $M_n$, establishing that it lies between $1$ and $4.3954$ for all $n$. We ultimately conjecture that  $M_n$ converges to $1.25$ as $n$ grows, based on computational analysis and other intuitive arguments. Although the number of $n$-circuits has been explored quite thoroughly, this paper also provides complete proofs of some important results.
\end{abstract}

\begin{figure}[H]
    \centering
    \includegraphics[scale=0.08]{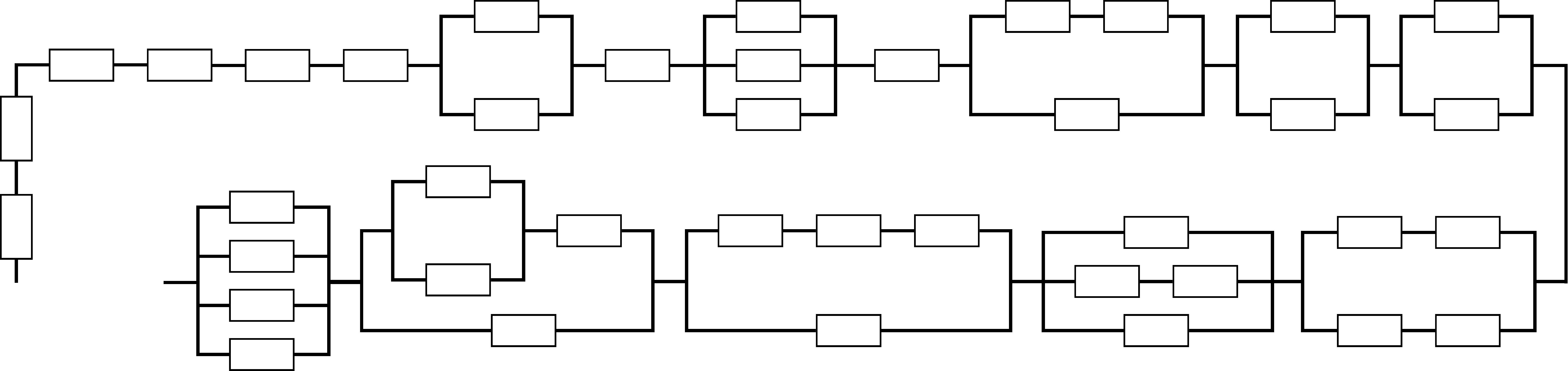}
    \caption{$\mathcal{O}_4$ \textemdash{} the $4$-omnicircuit.}
    \label{omni4}
\end{figure}
\newpage

\section*{Introduction}
The enumeration of series-parallel networks was first investigated by MacMahon more than a century ago \textemdash{} see \cite{mm2} and \cite{macmahon}. Since then, a few papers have been published on the topic, mainly concerned with enumeration and relations with Graph Theory. For example, there is a bijection between $n$-circuits and unlabeled trees with $n$ leaves in which no node has exactly one child \textemdash{} see \cite[A000084]{oeis}. What has yet to be fully investigated, however, is the main physical property: the resistances of such circuits.

It is well known that the total resistance of series circuits can be determined by summing the individual resistances of their components, whereas for parallel circuits this is achieved by reciprocating the sum of the reciprocals. The observation that series and parallel connections can be used to increase and decrease the total resistance of a circuit, respectively, naturally gives rise to the question:

\centerline{\emph{``What is the average resistance of all circuits}}
\centerline{\emph{consisting of a fixed number of identical resistors?''}}

What even further motivates this exploration is that were one willing to hand-calculate the first 6 values, they would notice an unpredictable behavior. We concentrate our investigation on $n$-circuits, which are series-parallel networks composed of exactly $n$ unit resistors. The graph readily suggests that the lower bound of the average resistance of all $n$-circuits, $M_n$, is $1$, a result that will be proven later in the paper.

\begin{figure}[H]
    \centering
    \includegraphics[scale=0.40]{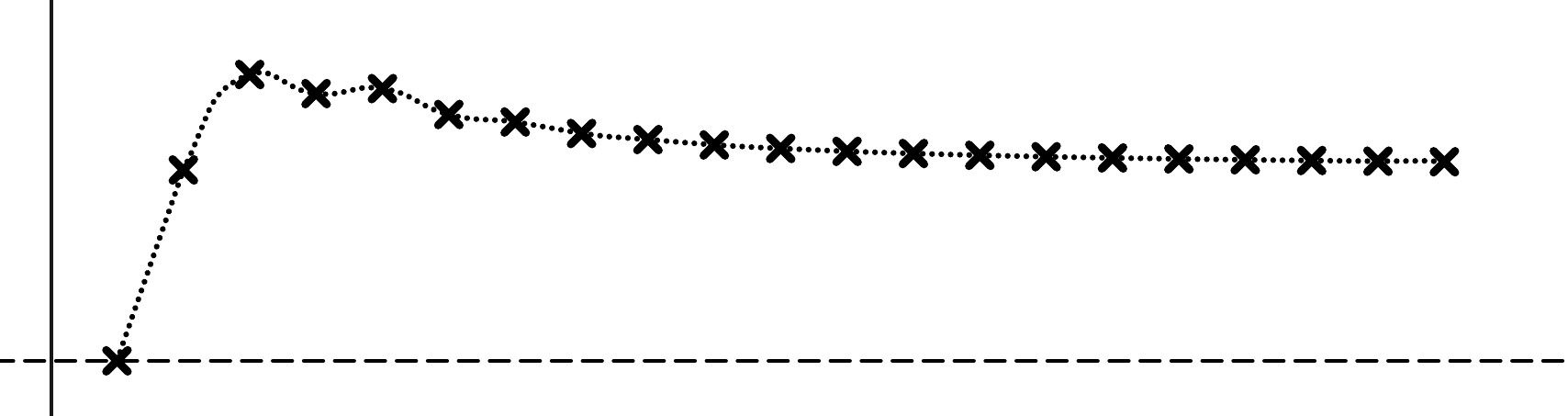}
    \caption{The first few values of $M_n$ connected by Bézier curves. The dashed line is $y = 1$.}
    \label{m}
\end{figure}

\newpage

\section*{Notation \& Definitions}

\begin{align*}
    \text{Series Circuit} \ \ \ \ &\coloneq \ \ \ \  \text{A series combination of smaller circuits} \\
    \text{Parallel Circuit} \ \ \ \ &\coloneq \ \ \ \  \text{A parallel combination of smaller circuits} \\
    \Gamma_n \ \ \ \ &\coloneq \ \ \ \  \text{The set of all }n\text{-circuits} \\
    Q_n \ \ \ \ &\coloneq \ \ \ \  \text{The number of }n\text{-circuits, }|\Gamma_n| \\
    R_n \ \ \ \ &\coloneq \ \ \ \  \text{The total resistance of all }n\text{-circuits} \\
    M_n \ \ \ \ &\coloneq \ \ \ \  \text{The average resistance of all }n\text{-circuits, }R_n/Q_n \\
    \Omega_n \ \ \ \ &\coloneq \ \ \ \  \text{The multiset of resistances of all }n\text{-circuits} \\ 
    \mathcal{O}_n \ \ \ \ &\coloneq \ \ \ \  \text{The }n\text{-omnicircuit} \text{, all }n\text{-circuits connected in series}  \\
    C_i(n) \ \ \ \ &\coloneq \ \ \ \  \text{The \# of appearances of any parallel }i\text{-circuit in }\mathcal{O}_n \\
    r(\gamma) \ \ \ \ &\coloneq \ \ \ \  \text{The resistance of the circuit }\gamma \\
    r_k(\gamma) \ \ \ \ &\coloneq \ \ \ \  \text{The } k\text{-resistance of the circuit }\gamma \\
    \gamma^* \ \ \ \ &\coloneq \ \ \ \  \text{The inverse of the circuit }\gamma \\
    \subC{\gamma} \ \ \ \ &\coloneq \ \ \ \  \text{The non-empty multiset in the OPoM notation of }\gamma \\ 
    \widetilde{\square{}} \ \ \ \ &\coloneq \ \ \ \  \square{}\text{ considering only series circuits} \\
    \overline{\overline{\square{}}} \ \ \ \ &\coloneq \ \ \ \  \square{}\text{ considering only parallel circuits} \\
    \mathcal{D}(\gamma) \ \ \ \ &\coloneq \ \ \ \  \text{The depth of the circuit } \gamma \\
    s^{\times{}} \ \ \ \ &\coloneq \ \ \ \  \text{The support set of the multiset }s \\
    v_s(i) \ \ \ \ &\coloneq \ \ \ \  \text{The multiplicity of }i\text{ in }s \\
    p \vdash n \ \ \ \ &\coloneq \ \ \ \  p\text{ is a partition of }n \\
    d(n) \ \ \ \ &\coloneq \ \ \ \  \text{The number of positive divisors of } n \\
\end{align*}
\newpage

\section{Quantifying Circuits}
    Let $Q_n$ denote the number of distinct $n$-circuits.
    To begin counting the $n$-circuits, one must come up with a way to uniquely represent them. We will be using Ordered Pairs of Multisets (OPoM) to do so. Let $(\{\}, \{\})$ represent the unit circuit, which is a single resistor and is considered to be both series and parallel. The series connection of a multiset of parallel circuits $\{\alpha_1, \alpha_2, \alpha_3, \cdots, \alpha_m \}$ is represented by $\alpha = (\{\alpha_1, \alpha_2, \alpha_3, \cdots, \alpha_m \}, \{\})$, and the parallel connection of a multiset of series circuits $\{\beta_1, \beta_2, \beta_3, \cdots, \beta_n \}$ is represented by $\beta = (\{\}, \{\beta_1, \beta_2, \beta_3, \cdots, \beta_n \})$. Not every ordered pair produces a circuit: a circuit  $\gamma = (\{m_1\}, \{m_2\})$ represents a series or a parallel circuit if and only if $m_1 = \{\}$ or $m_2 = \{\}$, respectively. This motivates the use of $ \subC{\gamma} $ to denote the non-empty multiset of sub-circuits of $\gamma$. The sub-circuits of any circuit are of its opposite connection type.
\begin{theorem}[Partition-based $n$-circuit Generation]
    The number of circuits generated by a partition $p$, $N(p)$ can be recursively counted by:
    \begin{equation*}
        N(p)=\prod_{i\in p^{\times}}\sum_{t\vdash v_p(i)} {{q_i}\choose{|t|}} |t|!\prod_{j\in t^{\times}} \frac{1}{v_t(j)!}
    \end{equation*}
    where
    \begin{equation*}
        Q_n=2q_n-\delta_{n1}=\sum_{p\vdash n} N(p)
    \end{equation*}
    counts the number of $n$-circuits $\forall n\in\mathbb{N}$.
\end{theorem}

\begin{proof}
    In order to construct a series $n$-circuit $\gamma = (\{\gamma_1, \gamma_2, \gamma_3, \cdots, \gamma_k\}, \{\})$, one should first choose the number of resistors in each $\gamma_i$, namely an unordered $k$-tuple of numbers adding up to $n$. Thus, one should start by considering a partition $p$ of $n$. For any unique element $i$ in $p$, $v_p(i)$ parallel $i$-circuits have to be chosen. However, these $v_p(i)$ parallel $i$-circuits are not necessarily unique, therefore, all the different ways for those circuits to be equivalent should be considered. Mathematically, partitions of $v_p(i)$ represent these cases, as any element of $t \vdash v_p(i)$ can be thought of as the number of times that a certain $i$-circuit appears in $\gamma$. Now any permutation of $|t|$ parallel $i$-circuits from the possible $q_i$ options can be considered alongside $t$ to build the circuit; except some cases are still over-counted: if the elements of $t$ corresponding to some of these $|t|$ parallel $i$-circuits are the same, then swapping them in the permutation does not change $\gamma$. This over-counting can be accounted for by dividing by $v_t(j)!$ for every unique $j$ in $t$. A similar argument holds for parallel $n$-circuits.
\end{proof}

\begin{table}[H]
    \centering
    \begin{tabular}{c|cccccccccccc}
        $n$ & 1 & 2 & 3 & 4 & 5 & 6 & 7 & 8 & 9 & 10 & 11 & 12\\
        \hline
        $Q_n$ & 1 & 2 & 4 & 10 & 24 & 66 & 180 & 522 & 1532 & 4624 & 14136 & 43930\\
    \end{tabular}
    \caption{$Q_n$ for $n \leq 12$}
    \label{tab:Q}
\end{table}
We define $q_0=Q_0=1$.

\section{Inverses}
The inverse of a circuit is defined as the circuit obtained by switching the order of every pair in its OPoM notation. This reveals a bijection between the series and parallel $n$-circuits, which is a property that will be used extensively. 

\begin{theorem}[Circuit-Multiplicative Inversion]
\label{inverse_resistance}
For any circuit $\gamma$,
\begin{equation*}
    r(\gamma)r(\gamma^*) = 1
\end{equation*}

\end{theorem}

\begin{proof}
    For any circuit $\gamma$, define its depth $\mathcal{D}(\gamma)$ to be the maximum nesting of its resistors. We will prove the theorem by strong induction on $\mathcal{D}(\gamma)$.
    \begin{itemize}
        \item \emph{Base Case}: $\mathcal{D}(\gamma) = 0 \Rightarrow \gamma$ is unit $\Rightarrow r(\gamma^*)=r(\gamma)=1 \Rightarrow r(\gamma^*)r(\gamma)=1$.

        \item \emph{Induction Hypothesis}: $r(\gamma)r(\gamma^*) = 1,\ \forall \mathcal{D}(\gamma) < n$.

        \item \emph{Induction Step}:\par Let $\mathcal{D}(\gamma) = n$. If $\gamma$ is a series circuit, let $ \subC{\gamma} = \{\gamma_1, \gamma_2, \gamma_3, \cdots, \gamma_k\}$, where $\mathcal{D}(\gamma_i) < \mathcal{D}(\gamma) = n,\ \forall i$. Then:
    \begin{equation*}
        \gamma^*=(\{\}, \{\gamma^*_1, \gamma^*_2, \gamma^*_3, \cdots, \gamma^*_k\}) \Rightarrow r(\gamma^*) =\frac{1}{\sum_{i=1}^k \frac{1}{r(\gamma_i^*)}}
    \end{equation*}
    By the induction hypothesis:
    \begin{equation*}
        r(\gamma_i^*) =\frac{1}{r(\gamma_i)},\ \forall i \Rightarrow r(\gamma^*) =\frac{1}{\sum_{i=1}^k r(\gamma_i)}=\frac{1}{r(\gamma)}.
    \end{equation*}
    The parallel case can be proven similarly.
    \end{itemize}
\end{proof}

\begin{corollary}[The Lower Bound]
\label{lowerbound}
    Let $M_n$ be the average resistance of all $n$-circuits. Then, $\forall n \in \mathbb{N}$:
    \begin{equation*}
        M_n \ge 1
    \end{equation*}
\end{corollary}

\begin{proof}
\begin{equation*}
    GM(\{r(\gamma):\gamma \in \Gamma_n\})=\left( \prod_{\gamma \in \widetilde{\Gamma}_n} r(\gamma)r(\gamma^*)\right)^\frac{1}{Q_n}=\left( \prod_{\gamma \in \widetilde{\Gamma}_n} 1 \right)^{\frac{1}{Q_n}}=1
\end{equation*}
Since the arithmetic mean is always greater than or equal to the geometric mean, this gives the lower bound $M_n\ge 1$ \textemdash{} equality only occurs at $n = 1$.
\end{proof}

\section{Properties of $C_i(n)$}
Define the $n$-omnicircuit to be the series connection of all $n$-circuits \textemdash{} see Figure \ref{omni4} \textemdash{} and $R_n$ to be the total resistance of all $n$-circuits. Mathematically:

\begin{equation*}
    \begin{split}
        \mathcal{O}_n &= \left(\left\{ \overline{\overline{\Gamma}}_n + \sum\limits_{\gamma \in \widetilde{\Gamma}_n} \subC{\gamma} \right\}, \{\}\right) \\
        R_n &= r(\mathcal{O}_n) = \sum_{\gamma \in \Gamma_n} r(\gamma)
    \end{split}
\end{equation*}

Omnicircuits prove to be the key to the investigation of $R_n$ as they motivate the introduction of $C_i(n)$. Due to symmetry, one can define $C_i(n) = v_{\subC{\mathcal{O}}_{n}}(\gamma)$, where $\gamma$ is any parallel $i$-circuit. In other words, $C_i(n)$ counts the number of appearances of any $i$-circuit in the $n$-omnicircuit. By definition, $C_n(n) = 1$, $ \forall n\in\mathbb{N}$, and $C_i(n)$ is not defined for $i > n$. This function is particularly useful as it allows for an alternative calculation of $R_n$:
\begin{equation}
\label{rsumrp}
    R_n = \sum_{i=1}^n \overline{\overline{R}}_iC_i(n)
\end{equation}

\begin{table}[H]
\centering
\begin{tabular}{c|cccccccccccc}
$C_{i}(n)$                    & 1 & 2 & 3 & 4 & 5  & 6  & 7   & 8   & 9   & 10   & 11   & 12     \\ 
\hline
1                           & 1 & 2 & 4 & 8 & 18 & 42 & 108 & 288 & 810 & 2342 & 6966 & 21102  \\
2                           &   & 1 & 1 & 3 & 5  & 13 & 29  & 79  & 209 & 601  & 1741 & 5225   \\
3                           &   &   & 1 & 1 & 2  & 5  & 11  & 26  & 71  & 191  & 548  & 1603   \\
4                           &   &   &   & 1 & 1  & 2  & 4   & 11  & 25  & 68   & 184  & 533    \\
5                           &   &   &   &   & 1  & 1  & 2   & 4   & 10  & 25   & 67   & 182    \\
6                           &   &   &   &   &    & 1  & 1   & 2   & 4   & 10   & 24   & 67     \\
7                           &   &   &   &   &    &    & 1   & 1   & 2   & 4    & 10   & 24     \\
8                           &   &   &   &   &    &    &     & 1   & 1   & 2    & 4    & 10     \\
9                           &   &   &   &   &    &    &     &     & 1   & 1    & 2    & 4      \\
10                          &   &   &   &   &    &    &     &     &     & 1    & 1    & 2      \\
11                          &   &   &   &   &    &    &     &     &     &      & 1    & 1      \\
12                          &   &   &   &   &    &    &     &     &     &      &      & 1     
\end{tabular}
    \caption{$C_{i}(n)$ for $n \le 12$}
    \label{tab:cni}
\end{table}

As demonstrated by Table \ref{tab:cni}, the sequence $\{C_i(n)\}_{i=1}^n$ seems to end with similar values for different $n$. This pattern is equivalent to $\{C_{n-i}(n)\}_{n=i+1}^\infty$ ``converging'', as apparent in Table \ref{tab:cnn-i}.

\begin{table}[H]
\centering
\begin{tabular}{l|lllllllllll||l}
$C_{n-i}(n)$ & 1 & 2 & 3 & 4 & 5  & 6  & 7   & 8   & 9   & 10   & 11   & $Q_i$  \\ 
\hline
0        & 1 & 1 & 1 & 1 & 1  & 1  & 1   & 1   & 1   & 1    & 1    & 1         \\
1        &   & 2 & 1 & 1 & 1  & 1  & 1   & 1   & 1   & 1    & 1    & 1         \\
2        &   &   & 4 & 3 & 2  & 2  & 2   & 2   & 2   & 2    & 2    & 2         \\
3        &   &   &   & 8 & 5  & 5  & 4   & 4   & 4   & 4    & 4    & 4         \\
4        &   &   &   &   & 18 & 13 & 11  & 11  & 10  & 10   & 10   & 10        \\
5        &   &   &   &   &    & 42 & 29  & 26  & 25  & 25   & 24   & 24        \\
6        &   &   &   &   &    &    & 108 & 79  & 71  & 68   & 67   & 66        \\
7        &   &   &   &   &    &    &     & 288 & 209 & 191  & 184  & 180       \\
8        &   &   &   &   &    &    &     &     & 810 & 601  & 548  & 522       \\
9       &   &   &   &   &    &    &     &     &     & 2342 & 1741 & 1532      \\
10       &   &   &   &   &    &    &     &     &     &      & 6966 & 4624    
\end{tabular}

    \caption{$C_{n-i}(n)$ for $i \le 10$ and $n \le 11$, alongside $Q_i$}
    \label{tab:cnn-i}
\end{table}
This motivates:

\begin{theorem}[Omnicircuit Meta-Counting Property]
    \label{omcp}
    $C_i(n)$ and $Q_n$ are related by the following expression:
    \begin{equation*}
        C_i(n)=C_i(n-i)+Q_{n-i}
    \end{equation*}
    where $i, n \in \mathbb{N}$ and $i \le n$.
\end{theorem}

\begin{proof}
Any parallel $i$-circuit appears in $\mathcal{O}_{n}$ either as a sub-circuit of or connected to an $(n-i)$-circuit. The first case, by definition, happens $C_i(n-i)$ times, and the second case happens $Q_{n-i}$ times, once for each $n$-circuit.
\end{proof}

\begin{corollary}The observed behavior of $C_i(n)$ can be expressed as:
    \label{omcp_n/2}
    \begin{equation*}
    \end{equation*}
\end{corollary}

\begin{proof} 
$i>\frac{n}{2} \Rightarrow i>n-i \Rightarrow C_i(n-i)=0$. By \ref{omcp}, $C_i(n) = C_i(n-i) + Q_{n-i} = Q_{n-i}$.
\end{proof}

\begin{corollary}
    \label{csumq}
    $C_i(n)$ can be fully expressed in terms of $Q$ by:
    \begin{equation*}
        C_i(n) = \sum_{k=1}^{\lfloor \frac{n}{i} \rfloor} Q_{n-ki}
    \end{equation*}
    with special case:
    \begin{equation*}
        C_i(ni) = \sum_{k=0}^{n-1} Q_{ki}
    \end{equation*}
\end{corollary}

\begin{proof}
Recursively apply \ref{omcp} to get the general form. Replace $n$ with $ni$ and re-index to get the special case.
\end{proof}

\begin{remark}
\label{C1}
    Setting $i = 1$, $C_1(n)=\sum\limits_{i=0}^{n-1} Q_i$, which is the cumulative sum of $Q_n$.
\end{remark}

\begin{lemma}
    \label{nqn}
    By double-counting the number of resistors in $\mathcal{O}_n$, 
    \begin{equation*}
        nQ_n = 2C_1(n) -\delta_{1n} + \sum_{i=2}^{n-1} iQ_iC_i(n),\ \forall n \in \mathbb{N}
    \end{equation*}
\end{lemma}

\begin{proof}
    There are $q_i$ parallel $i$-circuits, each containing $i$ resistors and appearing $C_i(n)$ times in $\mathcal{O}_n$. On the other hand, there are $Q_n$ circuits in $\mathcal{O}_n$, each containing $n$ resistors. Thus, $nQ_n = \sum\limits_{i=1}^n iq_iC_i(n) = \frac{1}{2}\sum\limits_{i=1}^n (2q_i)iC_i(n) = \frac{1}{2}\sum\limits_{i=1}^n{i(Q_i + \delta_{1i})C_i(n)} = \frac{1}{2}\sum\limits_{i=1}^n{iQ_iC_i(n)} + C_1(n) = 2C_1(n) -\delta_{1n} + \sum\limits_{i=2}^{n-1}{iQ_iC_i(n)}$.
\end{proof}

\section{The Generating Function of $Q_n$}
The fundamental difference in the nature of the relationships found between $C_i(n)$ and $Q_n$ in the previous section allows for a non-trivial recursive expression of $Q_n$. By \ref{csumq} and \ref{nqn}:
\begin{equation*}
    nQ_n = 2C_1(n) -\delta_{1n} + \sum_{i=2}^{n-1} iQ_iC_i(n) = -\delta_{1n} + 2\sum_{i=0}^{n-1} Q_i + \sum_{i=2}^{n-1} iQ_i\sum_{j=1}^{\lfloor \frac{n}{i} \rfloor} Q_{n-ij} 
\end{equation*}
\begin{equation}
    \label{qrecursive}
    \Rightarrow Q_n = \frac{1}{n} \left( -\delta_{1n}+2\sum_{i=0}^{n-1} Q_i + \sum_{i=2}^{n-1} {iQ_i\sum_{j=1}^{\lfloor \frac{n}{i} \rfloor} Q_{n-ij}} \right)
\end{equation}
This allows for $O(n^2)$ computation of $Q_n$ and $O(n^3)$ computation of all $Q_i$ up to $i=n$, improving the complexity from exponential time \textemdash{} using brute force or $N(p)$ \textemdash{} to polynomial.

\begin{theorem}
\label{qgf}
    Let the generating function of $Q_n$, $\mathcal{Q}(x)$ be defined as:
    \begin{equation*}
        \mathcal{Q}(x) = \sum_{n=0}^\infty Q_nx^n
    \end{equation*}
    Then, $\mathcal{Q}(x)$ satisfies:
    \begin{equation}
    \label{firstQ}
        \mathcal{Q}(x) = \exp{\left( \sum\limits_{n=1}^\infty{\frac{\mathcal{Q}(x^n)+x^n-1}{2n}}\right) }
    \end{equation}
    and equivalently,
    \begin{equation}
    \label{secondQ}
        \mathcal{Q}(x) = \prod\limits_{n=1}^\infty{(1-x^n)^{-q_n}}
    \end{equation}
\end{theorem}
\begin{proof}
    This proof is quite technical and can be skipped; nonetheless, it is included for the sake of completeness.
    
    Applying (\ref{qrecursive}) inside the generating function gives:
    \begin{equation*}
        \begin{split}
            \sum_{n=0}^\infty {nQ_nx^n} =& \sum_{n=0}^\infty{\left( -\delta_{1n} + 2\sum_{i=0}^{n-1} Q_i + \sum_{i=2}^{n-1} {iQ_i\sum_{j=1}^{\lfloor \frac{n}{i} \rfloor} Q_{n-ij}} \right)x^n } \\
            \Rightarrow x\sum_{n=0}^\infty {Q_n\left(nx^{n-1}\right)} =& -x + 2\sum_{n=0}^\infty{\sum_{i=0}^{n-1}{Q_ix^n}} + \sum_{n=0}^\infty{\sum_{i=2}^{n-1}{\sum_{j=1}^{\lfloor \frac{n}{i} \rfloor}{iQ_iQ_{n-ij}x^n}}}
        \end{split}
    \end{equation*}
    
    By re-indexing and introducing $\mathcal{Q}'(x)$, one can get the following differential equation:
    \begin{equation*}
        \begin{split}
            x\mathcal{Q}'(x) =& -x + 2\sum_{i=0}^\infty{\sum_{n=i+1}^\infty{Q_ix^n}} + \sum_{i=2}^\infty{\sum_{n=i+1}^\infty{\sum_{j=1}^{\lfloor \frac{n}{i} \rfloor}{iQ_iQ_{n-ij}x^n}}} \\
            =& 2\sum_{i=0}^\infty{Q_i \frac{x^{i+1}}{1-x}} + \sum_{i=2}^\infty{\sum_{j=1}^\infty{\sum_{n=ij}^\infty{iQ_iQ_{n-ij}x^n}}} - \left( x + \sum_{i=2}^\infty{iQ_iQ_0x^i} \right)\\
            =& \frac{2x\mathcal{Q}(x)}{1-x} + \sum_{i=2}^\infty{iQ_i\sum_{j=1}^\infty{\sum_{n=0}^\infty{Q_{n}x^{n+ij}}}} - \sum_{i=0}^\infty{iQ_ix^i} \\
            \Rightarrow 2x\mathcal{Q}'(x) =& \frac{2x\mathcal{Q}(x)}{1-x} + \mathcal{Q}(x)\sum_{i=2}^\infty{iQ_i\sum_{j=1}^\infty{x^{ij}}}
        \end{split}
    \end{equation*}
    
    Trying to simplify $\sum_{i=2}^\infty{iQ_i\sum_{j=1}^\infty{x^{ij}}}$ using geometric series gives the Lambert Series of $Q_n$. Instead, one can introduce $\mathcal{Q}(x^j)$:
    \begin{equation*}
        \begin{split}
            2\frac{\mathcal{Q}'(x)}{\mathcal{Q}(x)} =& \frac{2}{1-x} + \frac{1}{x}\sum_{j=1}^\infty{\sum_{i=2}^\infty{iQ_i\left(x^j\right)^i}} \\
            =&  \frac{1}{1-x} + \sum_{j=1}^\infty{\sum_{i=1}^\infty{x^{j-1}iQ_i\left(x^j\right)^{i-1}}} \\
            =&  \frac{1}{1-x} + \sum_{j=1}^\infty{\frac{1}{j}\frac{d}{dx}\mathcal{Q}(x^j)}
        \end{split}
    \end{equation*}
    
    Then, integrating and using $\ln(1-x) = -\sum_{n=1}^\infty {\frac{x^n}{n}}$:
    \begin{equation*}
        \int_0^x{\frac{\mathcal{Q}'(t)}{\mathcal{Q}(t)}\ dt} = \frac{1}{2}\int_0^x{\frac{1}{1-t} + \sum_{j=1}^\infty{\frac{1}{j}\frac{d}{dt}\mathcal{Q}(t^j)} \  dt}
    \end{equation*}
    \begin{equation*}
        \begin{split}
            \Rightarrow \ln\mathcal{Q}(x) =& \frac{1}{2}\left( -\ln(1-x) + \sum_{j=1}^\infty{\frac{1}{j}\left( \mathcal{Q}(x^j) - \mathcal{Q}(0) \right)} \right) \\
            =& \sum\limits_{j=1}^\infty{\frac{\mathcal{Q}(x^j)+x^j-1}{2j}} \\
            \Rightarrow \mathcal{Q}(x) =& \exp{\left( \sum\limits_{j=1}^\infty{\frac{\mathcal{Q}(x^j)+x^j-1}{2j}} \right)}
        \end{split}
    \end{equation*}
    
    Finally, the equivalence between (\ref{firstQ}) and (\ref{secondQ}) can be shown as follows:
    \begin{equation*}
    \begin{split}
        \ln\mathcal{Q}(x) =& \sum\limits_{j=1}^\infty{\frac{\mathcal{Q}(x^j)+x^j-1}{2j}} \\
        =& \sum\limits_{j=1}^\infty{\frac{1}{2j}\left( 1 + x^j + \sum_{k=2}^\infty{Q_kx^{jk}} + x^j - 1 \right)} \\
        =& \sum\limits_{j=1}^\infty{\left( \frac{x^j}{j} + \frac{1}{2j}\sum_{k=2}^\infty{Q_kx^{jk}}\right)} \\
        \Leftrightarrow \ln\mathcal{Q}(x) + \ln(1-x) =& \frac{1} {2}\sum\limits_{j=1}^\infty{\sum_{k=2}^\infty{Q_k\frac{x^{jk}}{j}}} = \frac{1}{2}\sum\limits_{k=2}^\infty{Q_k\sum_{j=1}^\infty{\frac{x^{jk}}{j}}} \\
        =& \frac{-1}{2}\sum\limits_{k=2}^\infty{Q_k\ln(1-x^k)} \\
        \Leftrightarrow (1-x)\mathcal{Q}(x) =& \prod\limits_{k=2}^\infty{(1-x^k)^{-\frac{Q_k}{2}}} \Leftrightarrow \mathcal{Q}(x) = \prod\limits_{k=1}^\infty{(1-x^k)^{-q_k}}
    \end{split}
    \end{equation*}
\end{proof}

The two relations established in \ref{qgf} \textemdash{} (\ref{firstQ}) \& (\ref{secondQ}) \textemdash{} were derived by \cite{golinelli1997asymptoticbehaviortwoterminalseriesparallel} who used them to derive (\ref{qrecursive}) as well. The paper also derives the asymptotic behavior of $Q_n$, which proves to be crucial for our investigation; Golinelli achieved this by analyzing $\mathcal{Q}(x)$, concluding:
\begin{equation}
\label{asympq}
    Q_n \sim cd^{n}n^{-\frac{3}{2}}
\end{equation}
where $d = 3.5608393095389433 \pm 10^{-16}$ is the reciprocal of a root of $\mathcal{Q}(x) = 2$.

This result can also be derived using numerical extrapolation on (\ref{qrecursive}), but a lot of data is required \textemdash{} analyzing the first $2500$ terms gives a result of $d \approx 3.559$.
\newpage
\begin{theorem}
\label{c/q}
    The asymptotic behavior of $C_i(n)$ relative to $Q_n$ is:
    \begin{equation*}
        \lim_{n \rightarrow \infty} \frac{C_i(n)}{Q_n} = \frac{1}{d^i-1}, \text{\space\space\space\space\space} \forall i \in \mathbb{N} 
    \end{equation*}
\end{theorem}
\begin{proof}
    Using \ref{csumq} and (\ref{asympq}):
    \begin{equation*}
        \begin{split}
            \lim_{n \rightarrow \infty} \frac{C_i(n)}{Q_n} &= \lim_{n \rightarrow \infty} \frac{\sum_{j=1}^{\lfloor \frac{n}{i} \rfloor}{Q_{n-ij}}}{Q_n} \\
            &= \lim_{n \rightarrow \infty} \frac{\sum_{j=1}^{\lfloor \frac{n}{i} \rfloor}{cd^{n-ij}(n-ij)^{-\frac{3}{2}}}}{cd^nn^{-\frac{3}{2}}} \\
            &= \lim_{n \rightarrow \infty} \sum_{j=1}^{\lfloor \frac{n}{i} \rfloor}{d^{-ij}\left(\frac{n}{n-ij} \right)^{\frac{3}{2}}} \\
        \end{split}
    \end{equation*}
    Considering a constant $j$, as $n \rightarrow \infty$, $d^{-ij}$ does not change and $\left(\frac{n}{n-ij}\right)^{\frac{3}{2}}$ approaches $1$. The terms that are added due to $n$ increasing \textemdash{} generally when $j=an$ with $0 < a \le 1$ \textemdash{} are of the form $d^{-ani}\left(\frac{n}{n(1-ai)}\right)^{\frac{3}{2}}$ and will clearly approach $0$. Therefore, it is safe to assume that removing $\left(\frac{n}{n-ij}\right)^{\frac{3}{2}}$ from the sum will not change the limit:
    \begin{equation*}
            \lim_{n \rightarrow \infty} \frac{C_i(n)}{Q_n} = \lim_{n \rightarrow \infty} \sum_{j=1}^{\lfloor \frac{n}{i} \rfloor}{d^{-ij}} = \sum_{j=1}^{\infty}{d^{-ij}} = \frac{d^{-i}}{1-d^{-i}} = \frac{1}{d^i-1}
    \end{equation*}
\end{proof}

\section{Resistance: Circuits \& Biscuits}
Let $\widetilde{\Omega}_n = \{r(\gamma): \gamma \in \widetilde{\Gamma}_ n\}$ and $\overline{\overline{\Omega}}_n = \{r(\gamma): \gamma \in \overline{\overline{\Gamma}}_ n\}$ be the sets of resistances of all series and parallel $n$-circuits, respectively. We begin investigating $R_n$ by exploring the ranges of these two sets.

\begin{theorem}[Resistance Range]
    \label{rrange}
    The ranges of $\widetilde{\Omega}_n$ and $\overline{\overline{\Omega}}_n$, $\forall n>1$ can be expressed as:
    \begin{equation*}
        \frac{4}{n} \le r(\gamma) \le n,\  \forall \gamma \in \widetilde{\Gamma}_n,\ \ \ \ \frac{1}{n} \le r(\gamma) \le \frac{n}{4},\  \forall \gamma \in \overline{\overline{\Gamma}}_n.
    \end{equation*}
\end{theorem}

\begin{proof}
    It is clear that $r(\gamma) \le n,\  \forall \gamma \in \widetilde{\Gamma}_n$. It immediately follows from \ref{inverse_resistance} that $\frac{1}{n} \le r(\gamma),\  \forall \gamma \in \overline{\overline{\Gamma}}_n$. Thus, if $\widetilde{\gamma}_{min} = (\{\gamma_1, \gamma_2, \gamma_3, \cdots, \gamma_k\}, \{\})$ is the series $n$-circuit with the least resistance, then $r(\gamma_i) = \frac{1}{\alpha_i}$,\  $\forall i$, where $\alpha_i$ is the number of resistors in $\gamma_i$. This gives:
    \begin{equation*}
        \sum_{i=1}^k \alpha_i = n,\ \ \ \ \ \ r(\widetilde{\gamma}_{min}) = \sum_{i=1}^k \frac{1}{\alpha_i}
    \end{equation*}
    $r(\widetilde{\gamma}_{min})$ can be found by analyzing its partial derivatives with respect to $\alpha_i$:
    \begin{equation*}
        \begin{split}
            r(\widetilde{\gamma}_{min}) &= \sum_{j=1}^k \frac{1}{\alpha_i} = \frac{1}{\alpha_i} + \sum_{1 \le j \le k, j\ne i} \frac{1}{\left(n - \sum\limits_{1 \le l \le k, l \ne i, j} \alpha_l\right) - \alpha_i} \\
            \Rightarrow \frac{\partial r(\widetilde{\gamma}_{min})}{\partial \alpha_i} &= - \frac{1}{\alpha_1^2} + \sum_{1 \le j \le k, j\ne i} -\frac{-1}{\left( \left(n - \sum\limits_{1 \le l \le k, l \ne i, j} \alpha_l\right) - \alpha_i \right)^2} \\
            &= -\frac{1}{\alpha_i^2} + \sum_{1 \le j \le k, j\ne i} \frac{1}{\alpha_j^2} = -\frac{2}{\alpha_i^2} + \sum_{j=1}^k \frac{1}{\alpha_j^2}
        \end{split}
    \end{equation*}
    The minimum occurs when $\frac{\partial r(\widetilde{\gamma}_{min})}{\partial \alpha_i} = 0,\ \forall i$. In which case:
    \begin{equation*}
        \begin{split}
            \sum_{j=1}^k \frac{1}{\alpha_j^2} = \frac{2}{\alpha_i^2},\ \forall i 
            &\Rightarrow\sum_{i=1}^k{\sum_{j=1}^k \frac{1}{\alpha_j^2}} = \sum_{i=1}^k{\frac{2}{\alpha_i^2}} \\
            \Rightarrow\sum_{j=1}^k{\sum_{i=1}^k \frac{1}{\alpha_j^2}} = \sum_{i=1}^k{\frac{2}{\alpha_i^2}} 
            &\Rightarrow k\sum_{j=1}^k \frac{1}{\alpha_j^2} = 2\sum_{i=1}^k{\frac{1}{\alpha_i^2}} \\
            \Rightarrow k &= 2
        \end{split}
    \end{equation*}
    Thus, $\widetilde{\gamma}_{min} = (\{\gamma_1, \gamma_2\}, \{\})$ and  $r(\widetilde{\gamma}_{min}) = \frac{1}{\alpha_1} + \frac{1}{\alpha_2}$, with $\alpha_1 + \alpha_2 = n$. Note:
    \begin{equation*}
        \frac{1}{\alpha_1} + \frac{1}{\alpha_2} = \frac{\alpha_1 + \alpha_2}{\alpha_1\alpha_2} = \frac{\alpha_1 + \alpha_2}{\frac{1}{4}\left((\alpha_1 + \alpha_2)^2-(\alpha_1 - \alpha_2)^2\right)} = \frac{4n}{n^2-(\alpha_1 - \alpha_2)^2}
    \end{equation*}
    Thus $r(\widetilde{\gamma}_{min})$ is minimized when $\alpha_1$ and $\alpha_2$ are as close to each other as possible. For even $n$, this happens when $\alpha_1 = \alpha_2 = \frac{n}{2}$, and for odd $n$ when $\{\alpha_1, \alpha_2\} = \{\frac{n-1}{2}, \frac{n+1}{2}\}$. Because $\frac{4n}{n^2-1} > \frac{4n}{n^2} \ge \frac{4}{n}$, in either case $r(\widetilde{\gamma}_{min}) \ge \frac{4}{n}$. By \ref{inverse_resistance}, $r(\overline{\overline{\gamma}}_{max}) \le \frac{n}{4}$ as well. 
\end{proof}

\begin{remark}
    The nature of the operations alongside \ref{rrange} strongly suggests the inequality $\overline{\overline{R}}_n \le \widetilde{R}_n$, apparent from numerical analysis as well \textemdash{} see Appendix.
\end{remark}

Applying common analysis tools to understand $R_n$ begins to fail at about this point, motivating the simplification of the system. We propose the concept of \textbf{biscuits}, a portmanteau of \textbf{bi}nary cir\textbf{cuits}. Define an $n$-biscuit as a circuit containing $n$ identical resistors, recursively constructed by applying one of the two operations of connecting a resistor in series or parallel. It is immediately apparent that there are no more than $2^{n-1}$ unique $n$-biscuits.

\begin{table}[H]
\centering
\begin{tabular}{c|c|c}
    Quantity & $n$-circuits & $n$-biscuits \\ 
    \hline
    Count & $Q_n = 2q_n-\delta_{1n}$ & $\frak{Q}_n = 2\frak{q}_n-\delta_{1n}$ \\
    Total Resistance & $R_n$ & $\frak{R}_n$ \\
    Average Resistance & $M_n$ & $\frak{M}_n$ \\
    Circuits & $\Gamma_n$ & $\mathcal{B}_n$ \\
    Resistances & $\Omega_n$ & $\Psi_n$ \\

\end{tabular}
    \caption{Symbols for analogous quantities of $n$-circuits and $n$-biscuits.}
    \label{tab:biscuits}
\end{table}

\begin{theorem}[Biscuit Injection and Inversion]
    \label{injection}
    Let $\widetilde{\varphi}: \Psi_n \rightarrow \widetilde{\Psi}_{n+1}$ and $\overline{\overline{\varphi}}: \Psi_n \rightarrow \overline{\overline{\Psi}}_{n+1}$ represent adding a resistor in series and parallel, respectively. Then, they are both injections. Furthermore, 
    \begin{equation*}
        \widetilde{\varphi}\left(r\right)\overline{\overline{\varphi}}\left(\frac{1}{r}\right) = 1
    \end{equation*}
\end{theorem}
\begin{proof}
    First note that the two functions can be expressed as follows:
    \begin{equation*}
        \begin{split}
            \widetilde{\varphi}\left(\frac{a}{b}\right) &= \frac{a}{b} + 1 = \frac{a+b}{b} \\
            \overline{\overline{\varphi}}\left(\frac{a}{b}\right) &= \frac{1}{\frac{1}{\frac{a}{b}} + \frac{1}{1}}= \frac{a}{a+b}
        \end{split}
    \end{equation*}

    So $\widetilde{\varphi}\left(r\right)\overline{\overline{\varphi}}\left(\frac{1}{r}\right) = 1$ follows immediately. By way of contradiction assume that $\exists r_1, r_2 \in \Psi_n. r_1 \ne r_2,  \widetilde{\varphi}(r_1) = \widetilde{\varphi}(r_2)$. Writing $r_1 = \frac{a}{b}$ and $r_2 = \frac{c}{d}$, 
    \begin{equation*}
        \frac{a+b}{b} = \frac{c + d}{d} \Rightarrow \frac{a}{b} = \frac{c}{d}
    \end{equation*}
    which is a contradiction. A similar argument can be used for the injectivity of $\overline{\overline{\varphi}}_n$. Furthermore, a bijection is revealed, analogous to \ref{inverse_resistance}.
\end{proof}

\begin{corollary}
[Uniqueness]
    $\Psi_n$, the multiset of resistances of all $n$-biscuits does not include any duplicates. That is, $n$-biscuits have unique resistances. 
\end{corollary}
\begin{proof}
    Using induction with \ref{injection}, the uniqueness of elements of $\widetilde{\Psi}_n$ and $\overline{\overline{\Psi}}_n$ can be proved. It is further apparent that these two sets can only overlap when an element has resistance $1$, something that never happens after $n = 1$ because $\frac{a}{a+b}$ and $\frac{a+b}{b}$ cannot have equal numerators and denominators \textemdash{} $a, b > 0$. This further implies that $\mathcal{B}_n$ contains exactly $2^{n-1}$ elements.
\end{proof}
\newpage
\begin{theorem}[Biscuit Means]
    \label{frakm}
    The average resistance of all $n$-biscuits is given by:
    \begin{equation*}
        \frak{M}_n=\frac{3}{2}-2^{-n}, \text{\space\space\space\space\space} \forall n \in \mathbb{N}
    \end{equation*}
    Moreover , 
    \begin{equation*}
        \widetilde{\frak{M}}_n=\frac{5}{2}-2^{1-n}, \text{\space\space\space\space\space} \overline{\overline{\frak{M}}}_n=\frac{1}{2}, \text{\space\space\space\space\space} \forall n>1
    \end{equation*}

\end{theorem}

\begin{proof}
    As $\overline{\overline{\Psi}}_n = \overline{\overline{\varphi}}(\Psi_{n-1}) = \overline{\overline{\varphi}}(\widetilde{\Psi}_{n-1}) \cup \overline{\overline{\varphi}}(\overline{\overline{\Psi}}_{n-1}),\ \forall n>1$, by \ref{injection} one can write:
        
    \begin{equation*}
        \begin{split}
            \overline{\overline{\frak{R}}}_n &= \sum_{r \in \widetilde{\Psi}_{n-1}} {\overline{\overline{\varphi}}(r)} + \sum_{r \in \overline{\overline{\Psi}}_{n-1}} {\overline{\overline{\varphi}}(r)} \\
            &= \sum_{r \in \widetilde{\Psi}_{n-1}}{\overline{\overline{\varphi}}\left(\frac{r}{1}\right) + \overline{\overline{\varphi}}\left(\frac{1}{r}\right)} \\
            &= \sum_{r \in \widetilde{\Psi}_{n-1}} {\frac{r}{r+1} + \frac{1}{r+1}}\\
            &= \frak{q}_{n-1}\\                
        \end{split}
    \end{equation*}
    Therefore:
    \begin{equation*}
        \overline{\overline{\frak{M}}}_n=\frac{\overline{\overline{\frak{R}}}_n}{\frak{q}_n}= \frac{\frak{q}_{n-1}}{\frak{q}_n}=\frac{1}{2}
    \end{equation*}
    Using this result, a recurrence relation for the average resistance of series $n$-biscuits can be derived as follows:
    \begin{equation*}
        \begin{split}
            \widetilde{\frak{R}}_n&=\frak{R}_{n-1}+\frak{Q}_{n-1}\\
            &=\widetilde{\frak{R}}_{n-1}+\overline{\overline{\frak{R}}}_{n-1}+2\frak{q}_{n-1}\\
            \Rightarrow 2\widetilde{\frak{M}}_n&=\widetilde{\frak{M}}_{n-1}+\frac{1}{2}+2\\
            \Rightarrow \widetilde{\frak{M}}_{n+1}&=\frac{1}{2}\widetilde{\frak{M}}_n+\frac{5}{4}, \text{\space\space\space\space\space} \widetilde{\frak{M}}_2=2
        \end{split}
    \end{equation*}
    Solving this recurrence relation results in:
    \begin{equation*}
        \widetilde{\frak{M}}_n=\frac{5}{2}-2^{1-n}, \text{\space\space\space\space\space} \forall n > 1
    \end{equation*}
    Finally, averaging over means gives:
    \begin{equation*}
        \frak{M}_n=\frac{\widetilde{\frak{M}}_n+\overline{\overline{\frak{M}}}_n}{2}=\frac{3}{2}-2^{-n}, \text{\space\space\space\space\space} \forall n>1
    \end{equation*}
    Coincidentally, the expression is also satisfied for $n=1$ as $\frak{M}_1=1=\frac{3}{2}-2^{-1}$.
\end{proof}
\begin{remark}
    It follows that $\forall n>1$:
    \begin{equation*}
        \frak{R}_n = \frac{3}{4}\times2^n - \frac{1}{2}, \text{\space\space\space\space\space} \widetilde{\frak{R}}_n = \frac{5}{8}\times2^{n} - \frac{1}{2}, \text{\space\space\space\space\space} \overline{\overline{\frak{R}}}_n =  \frac{1}{8}\times2^{n}
    \end{equation*}
\end{remark}
\begin{lemma}
    \label{inequality_lemma}
    For any two positive sequences $\{ a_i \}_{i=1}^m$ and $\{ b_j \}_{j=1}^n$, the following inequality holds:

    \begin{equation*}
        \frac{1}{mn}\sum_{i=1}^m{\sum_{j=1}^n {{\frac{1}{\frac{1}{a_i} + \frac{1}{b_j}}}} }  \le \frac{1}{\frac{1}{\frac{1}{m}\sum_{i=1}^m{a_i}} + \frac{1}{\frac{1}{n}\sum_{i=1}^n{b_i}}}
    \end{equation*}
    
\end{lemma}

\begin{proof}
    Start by considering the \emph{AM-HM} inequality for the sequence $\{ a_i + b_j\}_{i=1}^m$, where $j$ is constant:
    \begin{equation*}
        \frac{m}{\sum _{i=1}^m{\frac{1}{a_i + b_j}}} \le \frac{1}{m}\sum _{i=1}^m{a_i + b_j}
    \end{equation*}
    This inequality can be modified to get the result as follows:
    \begin{equation*}
        \begin{split}
            \frac{m}{\sum _{i=1}^m{\frac{1}{a_i + b_j}}} &\le \frac{1}{m}\sum _{i=1}^m{a_i + b_j} \\
            \Rightarrow \frac{m}{\frac{1}{m}\sum _{i=1}^m{a_i + b_j}} &\le \sum _{i=1}^m{\frac{1}{a_i + b_j}} \\
            \Rightarrow m - \sum _{i=1}^m{\frac{b_j}{a_i + b_j}} &\le m - \frac{b_jm}{\frac{1}{m}\sum _{i=1}^m{a_i + b_j}} \\
            \Rightarrow \sum _{i=1}^m{\left(1-\frac{b_j}{a_i + b_j}\right)} &\le \frac{\sum _{i=1}^m{(a_i + b_j)} - b_jm}{\frac{1}{m}\sum _{i=1}^m{a_i + b_j}} \\
            \Rightarrow \sum _{i=1}^m{\frac{a_i}{a_i + b_j}} &\le \frac{\sum _{i=1}^m{a_i}}{\frac{1}{m}\sum _{i=1}^m{a_i + b_j}} \\
            \Rightarrow \frac{1}{m}\sum _{i=1}^m{\frac{a_ib_j}{a_i + b_j}} &\le \frac{\frac{1}{m}\sum _{i=1}^m{a_ib_j}}{\frac{1}{m}\left(\sum _{i=1}^m{a_i}\right) + b_j} \\
            \Rightarrow \frac{1}{m}\sum _{i=1}^m{{\frac{1}{\frac{1}{a_i} + \frac{1}{b_j}}}} &\le \frac{1}{\frac{1}{\frac{1}{m}\sum_{i=1}^m{a_i}} + \frac{1}{b_j}}
        \end{split}
    \end{equation*}

    Clearly, the symmetric version of this inequality also holds. Namely:
    \begin{equation*}
        \frac{1}{n}\sum _{j=1}^n{{\frac{1}{\frac{1}{b_j} + \frac{1}{a_i}}}} \le \frac{1}{\frac{1}{\frac{1}{n}\sum_{j=1}^n{b_j}} + \frac{1}{a_i}}
    \end{equation*}
    The result follows immediately:
    \begin{equation*}
        \frac{1}{mn}\sum_{i=1}^m{\sum_{j=1}^n {{\frac{1}{\frac{1}{a_i} + \frac{1}{b_j}}}} } \le \frac{1}{m}\sum_{i=1}^m{\frac{1}{\frac{1}{\frac{1}{n}\sum_{j=1}^n{b_j}} + \frac{1}{a_i}}} \le \frac{1}{\frac{1}{\frac{1}{m}\sum_{i=1}^m{a_i}} + \frac{1}{\frac{1}{n}\sum_{j=1}^n{b_i}}}
    \end{equation*}
    
\end{proof}

We seek to establish a connection between circuits and biscuits that will later allow us to finally bound $M_n$. Define rearranging an $n$-biscuit as combining its single resistors in one or more depths to create sub-circuits without modifying the non-unit sub-circuits. Then:
\begin{theorem}[The Connection]
    \label{the_connection}
     Rearranging series and parallel $n$-biscuits, on average, decreases and increases their resistance respectively.
\end{theorem}
\begin{proof}
    Begin by noting that any $n$-biscuit has at most one non-unit sub-circuit; this is because, by definition, the construction of an $n$-biscuit is invertible. Hence, if one assumes by way of contradiction that there exists an $n$-biscuit with more than one non-unit sub-circuit, after reverting the construction of all of its unit sub-circuits, one can remove neither a series nor a parallel connection to a unit circuit, reaching a contradiction.
    
    Now consider a parallel $n$-biscuit $\beta$ with $k$ unit sub-circuits and a non-unit sub-circuit that is a series $(n-k)$-biscuit. As by \ref{rrange} the connection of the $k$ unit circuits in parallel is the smallest possible value for the resistance of a $k$-circuit, any first-depth rearrangement of this biscuit will result in a positive change in its total resistance. Inversely, any such rearrangement in a series biscuit will result in a negative change; this corresponds to a second-depth rearrangement of $\beta$. This argument can be recursively repeated to show that any odd-depth rearrangement of parallel $n$-biscuits results in a positive change in total resistance and any even-depth rearrangement in a negative; the converse is true for series circuits. It can be shown, using calculus, that the magnitude of change decreases as depth increases, resulting in the statement of the theorem. The rigorous proof is left as an exercise for the reader.
\end{proof}

We now have the required tools to return to the main topic of investigation, the average resistance of $n$-circuits.

\begin{theorem}
    \label{parallel_upper_bound}
    The average resistance of parallel $n$-circuits is bounded by
    \begin{equation*}
         \frac{1}{2} < \overline{\overline{M}}_n < \frac{5}{2}
    \end{equation*}
\end{theorem}

\begin{proof}
    Any parallel $n$-circuit can be thought of as a rearrangement of a parallel $n$-biscuit, and \ref{the_connection} shows that the average resistance of these rearrangements is higher than that of parallel $n$-biscuits, resulting in the lower bound.
    
    Using \ref{inequality_lemma} and setting $\{ a_i \}$ and $\{ b_i \}$ to $\Omega_k$ and $\widetilde{\Psi}_{n-k}$ for any $1 < k < n$, one can get:
    \begin{equation*}
        \begin{split}
            \frac{1}{Q_k\frak{q}_{n-k}}\sum_{i=1}^{Q_k}{\sum_{j=1}^{\frak{q}_{n-k}} {{\frac{1}{\frac{1}{\gamma_i} + \frac{1}{\beta_j}}}} }  &\le \frac{1}{\frac{1}{M_k} + \frac{1}{\widetilde{\frak{M}}_{n-k}}} = \frac{M_k\widetilde{\frak{M}}_{n-k}}{M_k+\widetilde{\frak{M}}_{n-k}} \\
            \Rightarrow \sum_{i=1}^{Q_k}{\sum_{j=1}^{\frak{q}_{n-k}} {{\frac{1}{\frac{1}{\gamma_i} + \frac{1}{\beta_j}}}} } &\le \frac{Q_k\frak{q}_{n-k}M_k\widetilde{\frak{M}}_{n-k}}{M_k+\widetilde{\frak{M}}_{n-k}}
        \end{split}
    \end{equation*}
    
    Summing the left hand side over all $k$ gives the total resistance of all first-depth rearrangements of all parallel $n$-biscuits except for the parallel connection of $n$ resistors which can be neglected. Note that this includes all parallel $n$-circuits except for the ones that require higher-depth rearrangements. By \ref{the_connection}, however, we know that the weighted average of these terms should be larger than the true average of parallel $n$-circuits. Thus:
    \begin{equation*}
        \begin{split}
            \overline{\overline{M}}_n &\le \frac{1}{\sum_{k=1}^{n-1}{Q_k\frak{q}_{n-k}}}\sum_{k=1}^{n-1}\frac{Q_k\frak{q}_{n-k}M_k\widetilde{\frak{M}}_{n-k}}{M_k+\widetilde{\frak{M}}_{n-k}} \\
            &= \frac{1}{\sum_{k=1}^{n-1}{Q_k2^{-k}}}\sum_{k=1}^{n-1}\frac{2^{-k}Q_kM_k\widetilde{\frak{M}}_{n-k}}{M_k+\widetilde{\frak{M}}_{n-k}} \\
            &< \frac{\sum_{k=1}^{n-1}\frac{5M_k}{2M_k+5}Q_k2^{-k}}{\sum_{k=1}^{n-1}{Q_k2^{-k}}} \le \max{\left\{\frac{5M_k}{2M_k+5}\right\}} < \frac{5}{2}
        \end{split}
    \end{equation*}
\end{proof}

\begin{theorem}[Convergence]
    The average resistance of $n$-circuits converges as $n$ grows. Specifically, 
    \begin{equation*}
        M_n \sim H_{n, \frac{3}{2}}
    \end{equation*}
    where $H_{n, m}$ is the $n$th generalized harmonic number of order $m$.    
\end{theorem}

\begin{proof}
    An inequality can be obtained by plugging \ref{parallel_upper_bound} in (\ref{rsumrp}):
    \begin{equation*}
        R_n = \sum_{i=1}^n \overline{\overline{R}}_iC_i(n) = \sum_{i=1}^n \overline{\overline{M}}_iq_iC_i(n) < \frac{5}{2}\sum_{i=1}^n q_iC_i(n)
    \end{equation*}
    Dividing both sides by $Q_n$ gives the result:
    \begin{equation*}
        \begin{split}
            M_n < \frac{5}{2}\sum_{i=1}^n \frac{q_iC_i(n)}{Q_n} \\
        \end{split} 
    \end{equation*}
    By \ref{c/q}:

    \begin{equation*}
        \frac{q_iC_i(n)}{Q_n} \rightarrow \frac{q_i}{d^i-1}
    \end{equation*}
    Then using (\ref{asympq}):
    
    \begin{equation*}
        \frac{q_i}{d^i-1} \rightarrow \frac{i^{-\frac{3}{2}}d^i}{d^i-1} = \frac{1}{i^{\frac{3}{2}}(1-d^{-i})}
    \end{equation*}
    Finally, due to the dominance of the exponential decay term:

    \begin{equation*}
        \frac{1}{i^{\frac{3}{2}}(1-d^{-i})} \rightarrow \frac{1}{i^{\frac{3}{2}}} 
    \end{equation*}
    whose sum over the natural numbers indeed converges \textemdash{} to $\zeta\left(\frac{3}{2}\right)$.
\end{proof}
The numerical value of the upper bound is $4.3954$ at $n=2500$.

\section{Final Remarks}
As all the lines of attack tried to gather more information about $M_n$ cannot
possibly be collected in a single paper, we will include the most promising ideas
for further investigation. The following ideas were thoroughly investigated, however, not exhausted. We strongly encourage readers to examine them in hopes of further progress on the limit of $M_n$.

\subsection{Generalization}
The \textit{harmonic operation} $\frac{1}{\frac{1}{x} + \frac{1}{y}}$ can be re-written as $(x^{-1} + y^{-1})^{-\frac{1}{1}}$, and $x + y = (x^1 + y^1)^{\frac{1}{1}}$. This motivates the following generalization, using a similar notion to the power mean: for any $k \in \mathbb{R}$, define the $k$-series connection of a multiset of $m$ unit or $k$-parallel circuits $\{\alpha_1, \alpha_2, \alpha_3, \cdots, \alpha_m \}$ to have a $k$-resistance, $r_k$, of $(r_k(\alpha_1)^k + r_k(\alpha_2)^k + r_k(\alpha_3)^k + \cdots + r_k(\alpha_m)^k)^{\frac{1}{k}}$, and the $k$-parallel connection of a multiset of $n$ unit or $k$-series circuits $\{\beta_1, \beta_2, \beta_3, \cdots, \beta_n \}$ to have a $k$-resistance of $(r_k(\beta_1)^{-k} + r_k(\beta_2)^{-k} + r_k(\beta_3)^{-k} + \cdots + r_k(\beta_n)^{-k})^{-\frac{1}{k}}$; with the \textit{electric circuit} case happening when $k=1$. As 

\begin{equation*}
    \left(\sum_{i=1}^m {r_k(\gamma_i)^{k}}\right)^{\frac{1}{k}} = \left(\sum_{i=1}^m {\left(\frac{1}{r_k(\gamma_i)}\right)^{-k}}\right)^{\frac{1}{k}} = \frac{1}{\left(\sum_{i=1}^m {\left(\frac{1}{r_k(\gamma_i)}\right)^{-k}}\right)^{-\frac{1}{k}}},
\end{equation*}
the main property of the system, \ref{inverse_resistance} remains true and can be proved using the same induction. The significance of this generalization, however, is its confirming of the intuitive convergence value of $M_{\infty} = 1.25$. By numerical analysis on the average $k$-resistance of all $n$-circuits, it is apparent that the limit seems to approach the value at $n = 2$ from above, suggesting that ``$2 = \infty$'', at least when it comes to circuits.

\subsection{More $C_i(n)$ Properties}
The following two theorems are derived in a similar fashion and give more insight into $C_i(n)$, which in many ways is the heart of this paper.

\begin{theorem}
    Although difficult to interpret its meaning,
    \begin{equation*}
        \sum_{i=1}^{n}{C_i(n)} = \sum_{i=1}^{n} {d(i)Q_{n-i}}
    \end{equation*}
    $\forall n \in \mathbb{N}$, where $d(n)$ counts the number of positive divisors of $n$.
\end{theorem}
\begin{proof}
    In the equation in \ref{csumq}, the terms in the sum appear when the argument of $Q$ is a multiple of $i$ away from $n$. When summing $C_i(n)$ over all $i$, therefore, each $Q_{n-i}$ appears exactly $d(i)$ times, resulting in the above equation.  
\end{proof}
\begin{theorem}
    A more general form of the remark on \ref{csumq} can be written as
    \begin{equation*}
        \sum_{k=1}^{i} C_i(n+k) = \sum_{k=0}^{n} Q_k
    \end{equation*}
    with the $i = 1$ case corresponding to the previous result.
\end{theorem}

\begin{proof} Using \ref{csumq} again, one can notice the division algorithm and merge the sums. A final re-indexing gives the desired result.
    \begin{equation*}
        \begin{split}
            &\sum_{k=0}^{i-1}C_i(n+k) = \\
            \sum_{k=0}^{i-1} \sum_{j=1}^{\lfloor \frac{n+k}{i} \rfloor} Q_{n+k-ij} =\sum_{k=0}^{i-1} \sum_{ij\le n+k}&Q_{n+k-ij} = \sum_{ij\le n+i-1}\sum_{k=0}^{i-1} Q_{n+k-ij} = \sum_{k=0}^{n+i} Q_k\\
            \Rightarrow \sum_{k=0}^{i-1} C_i(n-&i+k) = \sum_{k=0}^{n} Q_k = \sum_{k=1}^{i} C_i(n+k)\\
        \end{split}
    \end{equation*}
\end{proof}

\section*{Conclusion}
Introducing $C_i(n)$ allowed us to look at the previously investigated series-parallel networks in a new way, opening some new doors to exploring the behavior of the resistances. Simplifying down to biscuits allowed us to further analyze the desired system, leading us to the main result of
\begin{equation*}
    1 < \lim\limits_{n \rightarrow \infty} {M_n} < 4.3954
\end{equation*}
and the two other averages that we found along the way: by \ref{inverse_resistance}, the harmonic mean of all $n$-circuits approaches a number between $\frac{1}{4.3954} = 0.228$ and $1$, and their geometric mean is always $1$ by the argument used in \ref{lowerbound}. Nevertheless, we did not find an exact limit. Computationally acquired data suggests that $M_n$ tends towards $1.25$, and the investigation of the $k$-resistance of $n$-circuits seems promising and may provide further evidence for our conjecture.

\newpage
\bibliographystyle{unsrtnat}
\nocite{*}
\bibliography{main}

\section*{Acknowledgments}
We thank Dr. Honeywill for supporting us throughout our exploration.

\section*{Appendix}
\begin{table}[H]
\centering
\begin{tabular}{c||c|c|c|c|c}
    $n$ & $Q_n$ & $R_n$ & $M_n$ & $\widetilde{R}_n$ & $\overline{\overline{R}}_n$  \\ 
    \hline
    \hline
      1 & 1 & $1.00\times10^{0}$ & 1.000 & $1.00\times10^{0}$ & $1.00\times10^{0}$ \\ 
      2 & 2 & $2.50\times10^{0}$ & 1.250 & $2.00\times10^{0}$ & $0.50\times10^{0}$ \\ 
      3 & 4 &  $5.50\times10^{0}$ & 1.375 & $4.50\times10^{0}$ & $1.00\times10^{0}$ \\ 
      4 & 10 & $1.30\times10^{1}$ & 1.350 & $1.05\times10^{1}$ & $3.00\times10^{0}$ \\
      5 & 24 & $3.26\times10^{1}$ & 1.357 & $2.55\times10^{1}$ & $7.06\times10^{0}$ \\
      6  & 66 & $8.73\times10^{1}$ & 1.323 & $6.66\times10^{1}$ & $2.07\times10^{1}$ \\
      7 & 180 & $2.36\times10^{2}$ & 1.313 & $1.80\times10^{2}$ & $5.60\times10^{1}$ \\ 
      8 & 522 & $6.77\times10^{2}$ & 1.298 & $5.13\times10^{2}$ & $1.65\times10^{2}$ \\ 
      9 & 1532 & $1.98\times10^{3}$ & 1.290 & $1.49\times10^{3}$ & $4.85\times10^{2}$ \\ 
      10 & 4624 & $5.93\times10^{3}$ & 1.283 & $4.46\times10^{3}$ & $1.47\times10^{3}$  \\
      11 & 14136 & $1.81\times10^{4}$ & 1.278 & $1.36\times10^{4}$ & $4.50\times10^{3}$ \\
      12 & 43930 & $5.60\times10^{4}$ & 1.274 & $4.20\times10^{4}$ & $1.40\times10^{4}$ \\
      13 & 137908 & $1.75\times10^{5}$ & 1.272 & $1.31\times10^{6}$ & $4.40\times10^{4}$ \\ 
      14 & 437502 & $5.55\times10^{5}$ & 1.269 & $4.15\times10^{5}$ & $1.40\times10^{5}$ \\ 
      15 & 1399068 & $1.77\times10^{6}$ & 1.267 & $1.33\times10^{6}$ & $4.47\times10^{5}$ \\ 
      16 & 4507352 & $5.71\times10^{6}$ & 1.266 & $4.26\times10^{6}$ &  $1.44\times10^{6}$ \\
      17 & 14611576 & $1.85\times10^{7}$ & 1.264 & $1.38\times10^{7}$ & $4.68\times10^{6}$ \\
      18 & 47633486 & $6.02\times10^{7}$ & 1.263 & $4.49\times10^{7}$ & $1.53\times10^{7}$ \\
      19 & 156047204 & $1.97\times10^{8}$ & 1.262 & $1.47\times10^{8}$ & $5.00\times10^{7}$ \\ 
      20 & 513477502 & $6.48\times10^{8}$ & 1.261 & $4.83\times10^{8}$ &  $1.65\times10^{8}$ \\ 
      21 & 1696305728 & $2.14\times10^{9}$ & 1.261 & $1.59\times10^{9}$ & $5.44\times10^{8}$ \\ 

\end{tabular}
    \caption{The number of $n$-circuits, the total resistance of $n$-circuits, the average resistance of $n$-circuits, the total resistance of series $n$-circuits, and the total resistance of parallel $n$-circuits, up to $n=21$.}
    \label{tab:data}
\end{table}

\end{document}